\documentclass[10pt]{smfart}

\usepackage[T1]{fontenc}
\usepackage[english,francais]{babel}
\usepackage{latexsym,amscd}
\usepackage{amsmath,amsfonts,amssymb,mathrsfs}
\usepackage{enumerate,euscript}
\usepackage{amssymb,url,xspace,smfthm}
\input xy
\xyoption{all}

\newcommand{\BibTeX}{{\scshape Bib}\kern-.08em\TeX}
\newcommand{\T}{\S\kern .15em\relax }
\newcommand{\AMS}{$\mathcal{A}$\kern-.1667em\lower.5ex\hbox
        {$\mathcal{M}$}\kern-.125em$\mathcal{S}$}

\DeclareMathOperator{\Hom}{Hom}

\DeclareMathOperator{\Pic}{Pic}

\DeclareMathOperator{\rang}{rk}

\DeclareMathOperator{\Spec}{Spec}

\title{\'Equidistribution et diff\'erentiabilit\'e}
\date{\today}
\author{Huayi Chen}

\address{Universit\'e Paris Diderot --- Paris 7, Institut de math\'ematiques de Jussieu.}
\email{chenhuayi@math.jussieu.fr}

\begin{document}
\def\smfbyname{}

\begin{abstract}
On propose un crit\`ere d'\'equidistribution par la diff\'erentiabilit\'e de certains invariants arithm\'etiques. Combin\'e avec la m\'ethode de pentes et les mesures asymptotiques, ce crit\`ere donne une nouvelle d\'emonstration ``conceptuelle'' des r\'esultats d'\'equidistribution obtenus initialement via le principe variationnel.
\end{abstract}

\begin{altabstract}
We propose a criterion of equidistribution by the differentiability of certain arithmetic invariants. Combined with the slope method and the asymptotic measures, this criterion gives a new ``conceptual'' proof to equidistribution results originally obtained via the variation principle.
\end{altabstract}
\maketitle

\tableofcontents
\section{Introduction}

Dans les probl\`emes d'\'equidistribution de nature
arithm\'etique, on s'int\'eresse au comportement
asymptotique d'une suite de mesures d\'efinies par des
points alg\'ebriques dans une
vari\'et\'e projective sur un corps de nombres.
L'approche arakelovienne de ces probl\`emes est
propos\'ee par Szpiro, Ullmo et Zhang. Bas\'e sur le
principe variationnel introduit dans leur article
\cite{Szpiro_Ullmo_Zhang}, des d\'eveloppements dans cette direction ont
\'et\'e men\'es dans des travaux comme \cite{Ullmo98,Zhang98,Autissier01,Autissier06,Chambert-Loir06} etc.

Soient $K$ un corps de nombres et $\overline K$ une
cl\^oture alg\'ebrique de $K$. Soit $X$ une vari\'et\'e
arithm\'etique projective sur $\Spec K$. On fixe un
plongement $\sigma$ de $K$ dans $\mathbb C$. Si $x$ est
un point de $X$ \`a valeur dans $\overline K$, alors il
definit une mesure de probabilit\'e bor\'elienne
\begin{equation}
\label{Equ:mesure eta}
\eta_{x,\sigma}:=\frac{1}{[K(x):K]}
\sum_{\begin{subarray}{c}
\widetilde\sigma:K(x)\rightarrow\mathbb C\\
\widetilde\sigma|_K=\sigma
\end{subarray}}\delta_{\sigma(x)}\end{equation}
sur $X^{\mathrm{an}}_\sigma$, o\`u pour tout $y\in
X_\sigma^{\mathrm{an}}$, $\delta_y$ d\'esigne la mesure
de Dirac concentr\'ee en $y$. On consid\`ere maintenant
une suite $(x_n)_{n\geqslant 1}$ de points dans $
X(\overline K)$, dans le probl\`eme d'\'equidistribution, on cherche des conditions sous lesquelles la suite de
mesure $(\eta_{x_n,\sigma})_{n\geqslant 1}$ converge
faiblement, ou de fa\c{c}on \'equivalente, la suite
d'int\'egrales $\big(\int_{X_\sigma^{\mathrm{an}}
}f\,\mathrm{d}\eta_{x_n,\sigma}\big)_{n\geqslant 1}$
converge dans $\mathbb R$ pour toute fonction continue
$f$ sur $X_\sigma^{\mathrm{an}}$.

Certainement la suite $(x_n)_{n\geqslant 1}$ ne devrait
pas \^etre quelconque. Dans la pratique, on demande
souvent la convergence de certaine fonction de hauteur
\'evalu\'ee en ces points. Dans la litt\'erature, cette
fonction de hauteur peut \^etre construite soit en
utilisant la th\'eorie des hauteurs \`a la Weil, soit
par la th\'eorie d'Arakelov. Dans cet article, on
adopte le deuxi\`eme point de vue. Dans la th\'eorie
d'Arakelov, les fonctions de hauteurs sont d\'efinies
relativement aux fibr\'es inversibles ad\'elique. On ne
demande ici aucune condition de positivit\'e ni
de lissit\'e aux m\'etriques. Soit $\overline L$ un
fibr\'e inversible ad\'elique sur $X$. Pour tout point
alg\'ebrique $x\in X(\overline K)$, l'image
r\'eciproque $x^*\overline L$ est un fibr\'e
inversible ad\'elique sur $\Spec{ K(x)}$, o\`u $K(x)$
est le corps de d\'efinition de $x$. La hauteur de $x$
par rapport \`a $\overline L$ est par d\'efinition le
{\it degr\'e d'Arakelov} normalis\'e
$\widehat{\deg}_n(x^*\overline L)$. L'avantage de cette
approche est que la fonction de hauteur ainsi d\'efinie
est automatiquement additive par rapport \`a $\overline
L$, autrement dit, si $\overline L_1$ et $\overline
L_2$ sont deux fibr\'es inversibles hermitiens sur $X$,
alors
\[h_{\overline L_1\otimes\overline L_2}(x)
=h_{\overline L_1}(x)+h_{\overline L_2}(x).\]

Soit $\widehat{\Pic}(X)$ le groupe des classes
d'isomorphisme de fibr\'es inversibles hermitiens sur
$X$. Pour toute fonction continue $f$ sur $
X_\sigma^{\mathrm{an}}$, on d\'esigne par $\mathcal
O_\sigma(f)$ le fibr\'e inversible hermitien dont le
fibr\'e inversible sous-jacent est $\mathcal O_X$, dont
la m\'etrique en toute place autre que $\sigma$ et
$\overline{\sigma}$ est triviale, et tel que,
\begin{equation}\forall\, x\in X_{\sigma}^{\mathrm{an}}, \quad
\|\mathbf{1}_x\|_{\sigma}=e^{-f(x)},\\
\end{equation}
o\`u $\mathbf{1}$ est la section de l'unit\'e de
$\mathcal O_{X}$. L'application $f\mapsto\mathcal
O_\sigma(f)$ est en fait un homomorphisme du groupe
(additif) des fonctions continues sur
$X_\sigma^{\mathrm{an}}$ vers $\widehat{\Pic}(X)$. On
d\'esigne par $C_\sigma(X)$ l'image de cette
application. Si $\underline{x}=(x_n)_{n\geqslant 1}$
est une suite de points alg\'ebriques dans $X $, on
d\'esigne par $\varphi_{\underline{x}}:\widehat{\Pic}(X
)\rightarrow\mathbb R\cup\{\pm\infty\}$ la fonction
d\'efinie par $\varphi_{\underline{x}}(\overline L
)=\displaystyle\liminf_{n\rightarrow+\infty
}h_{\overline L}(x_n)$. Dans cet article, on \'etablit le
r\'esultat suivant:
\begin{quote}\it
\hskip\parindent On suppose que $\overline L$ est un
fibr\'e inversible ad\'elique tel que la suite
$(h_{\overline L}(x_n))_{n\geqslant 1}$ converge. Alors
la suite de mesure $(\eta_{x_n,\sigma})_{n\geqslant 1}$
converge faiblement si et seulement si la fonction
$\varphi_{\underline{x}}$ est diff\'erentiable en $\overline
L$ pour les directions dans $C_\sigma(X )$. En outre,
la mesure limite co\"{\i}ncide \`a la d\'eriv\'ee de
$\varphi_{\underline{x}}$.
\end{quote}
On pr\'ecisera dans \S\ref{Sec:critere} la d\'efinition de la
d\'erivabilit\'e d'une fonction sur $\widehat{\Pic}(X)$
et on expliquera que ce r\'esultat --- combin\'e avec
un argument de convexit\'e --- permet d'unifier certaines preuves par le principe
variationnel. En effet, pour \'etablir la
d\'erivabilit\'e de la fonction
$\varphi_{\underline{x}}$ en $\overline L$, il suffit
de la minorer par une fonction $\psi$ qui est
diff\'erentiable en $\overline L$ et telle que
$\varphi_{\underline{x}}(\overline L)=\psi(\overline
L)$. Ce crit\`ere est souple et demande aucune
condition de positivit\'e {\it a priori} sur les
m\'etrique du fibr\'e ad\'elique $\overline L$.

Un r\'esultat de Zhang \cite{Zhang95} permet de minorer $\varphi_{\underline{x}}(\overline L)$ par la hauteur normalis\'ee de $X$ par rapport \`a $\overline L$ lorsque la suite $\underline{x}$ est g\'en\'erique. C'est un point cl\'e dans les r\'esultats d'\'equidistribution via le principe variationnel. Cependant, il semble que l'\'egalit\'e entre $\varphi_{\underline{x}}(\overline L)$ et la hauteur normalis\'ee de $X$ est une hypoth\`ese tr\`es forte. Par exemple, elle implique que la mesure asymptotique de $\overline L$ est une mesure de Dirac, ou encore, le polygone de Harder-Narasimhan asymptotique est un morceau de droite. En outre, la diff\'erentiabilit\'e de la hauteur normalis\'ee de $X$ par rapport \`a $\overline L$ est un point difficile, qui demande ou bien une condition de positivit\'e \`a $\overline L$, ou bien des techniques analytiques comme par exemple le noyau de Bergman qui, me semble-t-il, n'a pas encore des analogues ultram\'etriques.

Il s'av\`ere que la m\'ethode de pentes \`a la Bost \cite{BostBour96,Bost2001,BostICM}, appliqu\'ee \`a l'application d'\'evaluation en ces points, permet de facilement minorer $\varphi_{\underline{x}}(\overline L)$ par la pente maximale asymptotique $\widehat{\mu}_{\max}^\pi(\overline L)$, qui est la limite des pentes maximales normalis\'ees des images directes des puissances tensorielles de $\overline L$. La pente maximale asymptotique est tr\`es li\'ee \`a un invariant arithm\'etique appel\'e la {\it mesure asymptotique} de $\overline L$, not\'ee $\nu_{\overline L}$ (voir \cite{Chen08,Chen_bigness,Chen_Fujita}). C'est une mesure de probabilit\'e bor\'elienne sur $\mathbb R$. Cet invariant est d\'efini pour tout les fibr\'es ad\'eliques hermitiens $\overline L$ tels que $L$ est gros. La pente maximale asymptotique est en fait la borne sup\'erieure du support de $\nu_{\overline L}$. Lorsque $\overline L$ est arithm\'etiquement gros, ou de fa\c{c}on \'equivalente, $\widehat{\mu}_{\max}^\pi(\overline L)>0$, on a les relations suivantes~:
\begin{equation}\label{Equ:inequ}\widehat{\mu}_{\max}(\overline L)\geqslant\int_{\mathbb R}\max\{x,0\}\,\nu_{\overline L}(\mathrm{d}x)\geqslant\int_{\mathbb R}x\,\nu_{\overline L}(\mathrm{d}x),\end{equation}
o\`u les deux int\'egrales sont respectivement \'egales au volume arithm\'etique (au sens de Moriwaki) normalis\'e et \`a la capacit\'e sectionnelle normalis\'ee. Cette derni\`ere s'identifie \`a la hauteur normalis\'ee de $X$ lorsque $\overline L$ satisfait \`a certaines conditions de positivit\'e. Le crit\`ere d'\'equidistribution mentionn\'e plus haut et la diff\'erentiabilit\'e de la fonction volume arithm\'etique \'etablie dans \cite{Chen_Diff} donnent une nouvelle d\'emonstration ``conceptuelle'' des r\'esultats d'\'equidistribution via le principe variationnel.

Le crit\`ere d'\'equidistribution sugg\`ere que, une meilleure compr\'ehension sur le domaine de diff\'erentiabilit\'e de la fonction $\widehat{\mu}_{\max}^\pi$ devrait conduire \`a des th\'eor\`emes d'\'equidistribution plus g\'en\'eral.

L'article est organis\'e comme la suite. Dans le deuxi\`eme paragraphe, on rappelle des notions concernant les fibr\'es ad\'eliques hermitiens. Dans le troisi\`eme paragraphe, on \'enonce et d\'emontre le crit\`ere d'\'equidistribution. Ensuite, on discute dans le quatri\`eme paragraphe les minorations de la limite inf\'erieure des hauteurs. Enfin, dans le cinqui\`eme paragraphe, comme applications on interpr\`ete les preuves via le principe variationnel par le langage de diff\'erentiabilit\'e et donne une nouvelle preuve.

\bigskip

\noindent{\bf Remerciements.} Les r\'esultats dans cet article ont \'et\'e pr\'esent\'es dans une rencontre de l'ANR ``Berkovich'', je tiens \`a remercier les organisateurs et les participants. Je suis reconnaissant \`a D. Bertrand, J.-B. Bost, A. Chambert-Loir, C. Gasbarri, C. Mourougane et C. Soul\'e pour de tr\`es int\'eressantes discussions et remarques.

\section{Rappels sur les fibr\'es inversibles ad\'eliques}

Dans cet article, le symbole $K$ d\'esigne un corps de
nombres et $\mathcal O_K$ d\'esigne la cl\^oture
int\'egrale de $\mathbb Z$ dans $K$. On note $\Sigma$
l'ensemble des places de $K$, qui s'\'ecrit comme
l'union disjointe de deux parties $\Sigma_f$ et
$\Sigma_\infty$, o\`u $\Sigma_f$ est l'ensemble des
places finies de $K$, qui s'identifie au spectre maximal de
$\mathcal O_K$; et $\Sigma_\infty$ est
l'ensemble des plongements de $K$ dans $\mathbb C$. On
fixe en outre une vari\'et\'e int\`egre et projective
$X$ d\'efinie sur $K$ et on d\'esigne par
$\pi:X\rightarrow\Spec K$ le morphisme structurel. Soit
$d=\dim(X)$.

\subsection{M\'etriques sur un fibr\'e inversible}

Soit $L$ un $\mathcal O_X$-module inversible. Soit
$\sigma\in\Sigma_\infty$. Une {\it m\'etrique}
(continue) sur $L$ en $\sigma$ est la donn\'ee, pour
tout ouvert $U$ de l'espace analytique
$X_\sigma^{\mathrm{an}}$ et toute section continue
$s\in C^0(U,L_{\sigma})$, d'une fonction continue
$\|s\|_\sigma:U\rightarrow\mathbb R_{\geqslant 0}$,
soumise aux conditions suivantes~:
\begin{enumerate}[1)]
\item pour toute fonction continue $a:U\rightarrow\mathbb
R$, on a $\|as\|_{\sigma}=|a|\cdot\|s\|_{\sigma}$;
\item si $x\in U$ est tel que $s(x)\neq 0$, alors $\|s\|_{\sigma}(x)\neq
0$.
\end{enumerate}
On note $n_\sigma=1$.

Soient $\mathfrak p\in\Sigma_f$ un id\'eal maximal de
$\mathcal O_K$ et $\mathbb F_{\mathfrak p}=\mathcal
O_K/\mathfrak p$ le corps r\'esiduel de $\mathfrak p$.
Soit $|\cdot|_{\mathfrak p}$ la valeur absolue sur $K$
telle que
\[\forall\,a\in K^{\times},\quad|a|_{\mathfrak p}=p^{-v_{\mathfrak p}(a)},\]
o\`u $v_{\mathfrak p}$ est la valuation discr\`ete sur
$K$ correspondant \`a $\mathfrak p$, $p$ est la
caract\'eristique de $\mathbb F_{\mathfrak p}$. Soit en
outre $n_{\mathfrak p}:=[\mathbb F_{\mathfrak
p}:\mathbb Z/p\mathbb Z]$. On d\'esigne par
$K_{\mathfrak p}$ le compl\'et\'e de $K$ par rapport
\`a la valeur absolue $|\cdot|_{\mathfrak p}$, et par
$\mathbb C_{\mathfrak p}$ le compl\'et\'e d'une
cl\^oture alg\'ebrique de $K_{\mathfrak p}$. La valeur
absolue $|\cdot|_{\mathfrak p}$ s'\'etend de fa\c{c}on
unique sur $\mathbb C_{\mathfrak p}$, et le corps
$\mathbb C_{\mathfrak p}$ est alg\'ebriquement clos et
complet pour la valeur absolue $|\cdot|_{\mathfrak p}$.
D'apr\`es \cite[3.4.1]{Berkovich90}, le foncteur de la
cat\'egorie $\mathbf{An}_{K_{\mathfrak p}}$ des espaces
analytiques sur $K_{\mathfrak p}$ au sens de Berkovich
(cf. \cite[\S 3.1]{Berkovich90}) vers la cat\'egorie
des ensembles, qui envoie tout espace analytique $Y$ en
l'ensemble des morphismes d'espaces annel\'es en
$K_{\mathbb p}$-alg\`ebres $\Hom_{K_{\mathfrak p}}(Y,X_{K_{\mathfrak p}})$, est
repr\'esentable par un $K_{\mathfrak p} $-espace
analytique que l'on notera $X^{\mathrm{an}}_{{\mathfrak
p}}$. Le $\mathcal O_X$-module inversible $L$
correspond \`a un module inversible sur l'espace
analytique $X_{\mathfrak p}^{\mathrm{an}}$ que l'on
notera $L_{\mathfrak p}$. Une {\it m\'etrique} sur $L$
en $\mathfrak p$ est alors la donn\'ee, pour toute
partie ouverte $U$ de $X_{\mathfrak p}^{\mathrm{an}}$
et toute section continue $s\in C^0(U,L_{\mathfrak
p})$, d'une fonction continue $\|s\|_{\mathfrak
p}:U\rightarrow\mathbb R_{\geqslant 0}$, soumise aux
conditions suivantes~:
\begin{enumerate}[1)]
\item pour toute fonction continue $a:U\rightarrow\mathbb C_{\mathfrak
p}$, on a $\|as\|_{\mathfrak p}=|a|_{\mathfrak
p}\cdot\|s\|_{\mathfrak p}$;
\item si $x\in U$ est tel que $s(x)\neq 0$, alors
$\|s\|_{\sigma}(x)>0$.
\end{enumerate}

Soit $Z$ un sous-sch\'ema ouvert de $\Spec\mathcal O_K$
qui contient $\mathfrak p$ et $\mathcal O_Z$ son
anneau. On appelle mod\`ele de $(X,L)$ sur $Z$ tout
couple $(\mathscr X,\mathscr L)$, o\`u $\mathscr X$ est
un $X$-sch\'ema projectif et plat tel que $\mathscr
X_K=X$ et $\mathscr L$ est un faisceau inversible sur
$\mathscr X$ tel que $\mathscr L_K=L $.

\subsection{Fibr\'e inversible ad\'elique}

On appelle {\it fibr\'e inversible ad\'elique} sur $X$
tout couple $\overline
L=(L,(\|\cdot\|_v)_{v\in\Sigma})$, o\`u $L$ est un
$\mathcal O_X$-module inversible, et o\`u
$\|\cdot\|_{v}$ est une metrique sur $L$ en $v$, soumis
aux conditions suivantes~:
\begin{enumerate}[1)]
\item pour toute place finie $\mathfrak p\in\Sigma_f$,
$\|\cdot\|_{\mathfrak p}$ est invariante sous l'action
du groupe de Galois $\mathrm{Gal}(\mathbb C_v/K_v)$;
\item les m\'etriques $(\|\cdot\|_{\sigma})_{\sigma\in\Sigma_\infty}$
sont invariantes par la conjugaison complexe;
\item il existe un sous-sch\'ema ouvert non-vide $Z$ de
$\Spec\mathcal O_K$ et un mod\`ele $(\mathscr
X,\mathscr L)$ de $(X,L)$ sur $Z$ tel que, pour toute place
$v\in Z\cap\Sigma$, la m\'etrique $\|\cdot\|_{v}$ soit
induite par le mod\`ele $(\mathscr X,\mathscr L)$.
\end{enumerate}
Si $\overline L$ est un fibr\'e inversible ad\'elique
sur $X$, on d\'esigne par $\pi_*\overline L$ l'espace
vectoriel $H^0(X,L)$ sur $K$, muni des m\'etriques sup.
C'est un {\it fibr\'e vectoriel ad\'elique} sur $\Spec
K$ (cf. \cite{Gaudron07}, voir aussi un rappel dans \S
\ref{SubSec:rappel}).

On d\'esigne par $\widehat{\mathrm{Pic}}(X)$ le groupe
des classes d'isomorphismes de fibr\'es inversibles
ad\'eliques sur $X$. Si $\mathfrak p$ est une place
finie et si $f$ est une fonction continue sur
$X_{\mathfrak p}^{\mathrm{an}}$, suppos\'ee \^etre
invariante par l'action du groupe de Galois
$\mathrm{Gal}(\mathbb C_\mathfrak p/K_\mathfrak p)$
lorsque $\mathfrak p$ est finie, on d\'esigne par
$\mathcal O_{\mathfrak p}(f)$ le fibr\'e inversible
ad\'elique sur $X$ dont le faisceau inversible
sous-jacent est $\mathcal O_X$ et tel que
\begin{gather*}\forall\,x\in X_{\mathfrak p}^{\mathrm{an}},
\quad\|\mathbf{1}\|_{\mathfrak p}(x)=(\#\mathbb
F_{\mathfrak p})^{- f(x)},\\
\forall\,v\in\Sigma\setminus\{\mathfrak p\},\,\forall\,
y\in X_v^{\mathrm{an}},\quad \|\mathbf{1}\|_v(y)=1.
\end{gather*}
De fa\c{c}on similaire, si $\sigma:K\rightarrow\mathbb
C $ est un plongement et si $f$ est une fonction
continue sur $X_\sigma^{\mathrm{an}}$, on d\'esigne par
$\mathcal O_{\sigma}(f)$ le fibr\'e inversible
ad\'elique dont le faiceau inversible sous-jacent est
$\mathcal O_X$, et tel que
\begin{gather*}\forall\,x\in X_{\sigma}^{\mathrm{an}},
\quad\|\mathbf{1}\|_{\sigma}(x)=e^{- f(x)},\\
\forall\,v\in\Sigma\setminus\{\sigma,\overline{\sigma}\},\,\forall\,
y\in X_v^{\mathrm{an}},\quad \|\mathbf{1}\|_v(y)=1.
\end{gather*}
Pour toute place $v$, l'application du groupe additif
$C^0(X_v^{\mathrm{an}})$ vers le groupe
$\widehat{\Pic}(X)$ qui envoie $f$ vers $\mathcal
O_v(f)$ est en fait un homomorphisme de groupes. On
d\'esigne par $C_v(X)$ l'image de cette application.

Soient $\overline L$ et $\overline{L'}$ deux fibr\'es
inversibles ad\'eliques sur $X$ et $f:L\rightarrow L'$
un homomorphisme non-nul. Pour toute place $v\in\Sigma$
et tout $x\in X_{v}^{\mathrm{an}}$, l'homomorphisme $f$
induit une application $\mathbb C_v$-lin\'eaire
$f_x:L_{v,x}\rightarrow L_{v,x}'$. On note
\[h_{v}(f):=\sup_{x\in X_v^{\mathrm{an}}}\log\|f_x\|_v,\]
appel\'e la {\it hauteur locale} de $f$ en $v$. On
d\'efinit en outre
\[h(f)=\frac{1}{[K:\mathbb Q]}\sum_{v\in\Sigma}n_vh_v(f).\]

Soit $\overline L$ un fibr\'e inversible ad\'elique sur
$X$. On dit qu'une section globale $s\in H^0(X,L)$ est
{\it effective} si, pour toute $v\in\Sigma$,
\[\|s\|_{v,\sup}:=\sup_{x\in X_v^{\mathrm{an}}}\|s\|_v(x)\leqslant 1.\]
On dit que $\overline L$ est {\it effectif} s'il admet
au moins une section effective non-nulle. Un fibr\'e
inversible ad\'elique $\overline L$ est effectif si et
seulement s'il existe un homomorphisme
$f:\overline{\mathcal O}_X\rightarrow\overline L$ tel
que $h_v(f)\leqslant 0$ pour toute $v$, o\`u
$\overline{\mathcal O}_X$ d\'esigne le fibr\'e
inversible ad\'elique trivial sur $X$.

\def\skip{On dit qu'un fibr\'e inversible

Soient $\overline L_1$ et $\overline L_2$ deux fibr\'es
inversibles hermitiens sur $X$ et
$\varphi:L_1\rightarrow L_2$ un homomorphisme non-nul.}

\section{Crit\`ere d'\'equidistribution par la diff\'erentiabilit\'e}
\label{Sec:critere}

On fixe une place $v$ dans $\Sigma$ et on d\'esigne par
$C_v(X)$ le sous-groupe de $\widehat{\Pic}(X)$ form\'e
des fibr\'es inversibles hermitiens de la forme
$\mathcal O_v(f)$, o\`u $f$ parcourt les fonctions
continues sur l'espace analytique $ X_v^{\mathrm{an}}$.

\subsection{Fonctions d\'erivables sur un semi-groupe}

Dans ce sous-paragraphe, on clarifie la notion de
d\'erivabilit\'e pour les fonctions sur
$\widehat{\Pic}(X)$. Par la soucis de lucidit\'e de
pr\'esentation, on
choisit de travailler dans un cadre g\'en\'eral d'un
groupe commutatitif $G$ dont la loi de groupe est
not\'e additivement.

Un {\it semi-groupe} dans $G$
est par d\'efinition un sous-ensemble $C$
de $G$ qui v\'erifie la condition suivante:
\begin{quote}
si $x$ et $y$ sont deux \'el\'ements dans
$C$, alors $x+y\in C$.
\end{quote}
Si $C$ est un semi-groupe dans $G$ et si $H$
est un sous-groupe de $G$, on dit que $C$ est
{\it ouvert} par rapport \`a $H$ si, pour
tout $x\in C$ et tout $y\in H$, il existe un
entier $n\geqslant 1$ tel que $nx+y\in C$, ou
de fa\c{c}on \'equivalente, $nx+y\in C$ pour
tout entier $n$ suffisamment positif.

Soient $H$ un sous-groupe de $G$ et $C$ un
semi-groupe de $G$ qui est ouvert par rapport
\`a $H$. Soit $f:C\rightarrow\mathbb R$ une
fonction, suppos\'ee \^etre {\it positivement homog\`ene}, i.e.,
pour tout $x\in C$ et tout entier $n\geqslant
1 $, on a $f(nx)=nf(x)$. Si $x\in C$ et si
$v\in H$, on dit que la fonction $f$ est {\it
d\'erivable} en $x$ le long de la direction
de $v$ si la suite
$(f(nx+v)-f(nx))_{n\geqslant 1}$ converge
dans $\mathbb R$. On dit que la fonction $f$
est {\it diff\'erentiable} en $x$ pour les
directions dans $H$ si elle est d\'erivable
en $x$ le long de toute $w\in H$, et si
l'application $D_xf:H\rightarrow\mathbb R$
qui associe $w\in H$ en
$\displaystyle\lim_{n\rightarrow\infty}f(nx+w)-f(nx)$
est additive, autrement dit,
$D_xf(u+w)=D_xf(u)+D_xf(w)$ quels que soit
$u,w\in H$. L'application $D_xf$ est
appel\'ee la {\it diffentielle} de $f$ en $x$
(relativement \`a $H$).

\begin{rema}\label{Rem:invariance de differentiabilite}
La diff\'erentiabilit\'e ne d\'epend pas du semi-groupe
de d\'efinition. On suppose que $C_1$ et $C_2$ sont
deux semi-groupes ouverts par rapport \`a $H$. Si $f_1$
(resp. $f_2$) est une fonction sur $C_1$ (resp. $C_2$)
dont les r\'estrictions \`a $C_1\cap C_2$ se
co\"{\i}ncident. Pour tout \'el\'ement $x\in C_1\cap
C_2$, la diff\'erentiabilit\'e de $f_1$ en $x$
pour les directions dans $H$ est \'equivalente \`a celle de
$f_2$ en $x$ pour les directions dans $H$.
\end{rema}

Le lemme suivant est utile dans la
d\'emonstration de la proposition
\ref{Pro:critere de differentiabilite}, qui
est un crit\`ere de diff\'erentiabilit\'e.

\begin{lemm}\label{Lem:critere diff}
Soit $H$ un sous-groupe de $G$. Si
$\varphi:H\rightarrow\mathbb
R\cup\{+\infty\}$ est une fonction
sur-additive (c'est-\`a-dire $\forall\,
x,y\in H,\,
\varphi(x+y)\geqslant\varphi(x)+\varphi(y)$)
et qui n'est pas identiquement infinie, et si
$\psi:H\rightarrow\mathbb R$ est une fonction
additive telle que $\varphi\geqslant\psi$,
alors $\varphi=\psi$.
\end{lemm}
\begin{proof}
Soit $\eta=\varphi-\psi$. La fonction $\eta$
est sur-additive sur $H$. De plus, elle est
positive. On d\'esigne par $\theta$
l'\'el\'ement neutre du groupe $H$. Si $x\in
H$ est un \'el\'ement tel que $\eta(x)>0$,
alors
$\eta(\theta)=\eta(x+(-x))\geqslant\eta(x)+\eta(-x)>0$.
En outre, la fonction $\eta$ n'est pas
identiquement infinie, il existe $y\in H$ tel
que $\eta(y)<+\infty$. Comme
$\eta(y)=\eta(y+\theta)\geqslant\eta(y)+\eta(\theta)$,
on obtient $\eta(\theta)\leqslant 0$. Cela
est absurde.
\end{proof}

\begin{prop}
\label{Pro:critere de differentiabilite} Soient $H$ un
sous-groupe de $G$, $C$ un semi-groupe dans $G$ qui est
ouvert par rapport \`a $H$, et $x\in C$. Si $f$ et $g$
sont deux fonctions r\'eelles positivement homog\`enes sur $C$ qui
satisfont aux conditions suivantes:
\begin{enumerate}[1)]
\item $\forall a,b\in C,\,f(a+b)\geqslant
f(a)+f(b)$,
\item $f\geqslant g$, $f(x)=g(x)$,
\item $g$ est diff\'erentiable en $x$ pour
les directions dans $H$,
\end{enumerate}
alors la fonction $f$ est diff\'erentiable en
$x$ pour les directions dans $H$. De plus, on
a $D_xf=D_xg$.
\end{prop}
\begin{proof}
Pour tout \'el\'ement $w\in H$, il existe
$n_0(w)\in\mathbb N_*$ tel que $nx+w\in C$
quel que soit $n\geqslant n_0(w)$. De plus,
la suite $(f(nx+w)-f(nx))_{n\geqslant
n_0(w)}$ est croissante. On d\'esigne par
$D_xf$ la fonction sur $H$ \`a valeurs dans
$\mathbb R\cup\{+\infty\}$ d\'efinie par
$D_xf(w)=\displaystyle\lim_{n\rightarrow\infty}
f(nx+w)-f(nx)$. Soit $\theta$ l'\'el\'ement
neutre de $H$. Comme $D_xf(\theta)=0$, la
fonction $D_xf$ n'est pas identiquement
infinie. En outre, les fonctions $f$ et $g$
sont homog\`enes, et $f(x)=g(x)$, donc
$f(nx)=g(nx)$ quel que soit $n\in\mathbb
N_*$. Par cons\'equent, on a $D_xf\geqslant
D_xg$, o\`u la fonction
$D_xg:H\rightarrow\mathbb R$ est d\'efinie
comme
$D_xg(w)=\displaystyle\lim_{n\rightarrow\infty}
g(nx+w)-g(nx)$, qui est additive car $g$ est
diff\'erentiable en $x$. Enfin, si $u$ et $w$
sont deux \'el\'ements dans $H$, alors
\[\begin{split}&\quad\;D_xf(u+w)=
\lim_{n\rightarrow\infty}f(2nx+u+w)-2nf(x)
\\&\geqslant\lim_{n\rightarrow\infty}f(nx+u)+f(nx+w)-2nf(x)
=D_xf(u)+D_xf(w).\end{split}\] D'apr\`es le
lemme \ref{Lem:critere diff}, on obtient que
$D_xf=D_xg$ est additive, donc la fonction
$f$ est diff\'erentiable en $x$.
\end{proof}

\subsection{Crit\`ere d'\'equidistribution}

Dans ce sous-paragraphe, on d\'emontre un crit\`ere
d'\'equidistribution. On consid\`ere une suite
$\underline{x}=(x_n)_{n\geqslant 1}$ de points
alg\'ebriques dans $X$. On dit que la suite
$\underline{x}$ satisfait \`a la condition
d'\'equidistribution en $v\in\Sigma$ si la suite de
mesure $(\eta_{x_n,v})_{n\geqslant 1}$ converge
faiblement, ou de fa\c{c}on \'equivalente, pour tout
fibr\'e inversible hermitien $\mathcal O_v(f)\in C_v(
X)$, la suite $(h_{\mathcal O_v(f)}(x_n))_{n\geqslant
1}$ converge dans $\mathbb R$. Pour tout fibr\'e
inversible ad\'elique $\overline L$ dans
$\widehat{\Pic}( X)$, on d\'esigne par
$\varphi_{\underline{x}}(\overline L)$ l'\'el\'ement
$\displaystyle\liminf_{n\rightarrow+\infty}h_{\overline
L }(x_n)$ dans $\mathbb R\cup\{\pm\infty\}$.
\begin{prop}\label{Pro:equidistribution equivalent a ladditive}
Les conditions suivantes sont \'equivalentes:
\begin{enumerate}[1)]
\item la suite $\underline{x}$ satisfait \`a
la condition d'\'equidistribution;
\item la restriction de
$\varphi_{\underline{x}}$ \`a $C_v(X)$ est additive.
\end{enumerate}
\end{prop}
\begin{proof} Pour tout fibr\'e inversible hermitien
$\mathcal O_v(f)$ dans $C_v(X)$, la suite $(h_{\mathcal
O_v(f)}(x_n))_{n\geqslant 1}$ est born\'ee. Donc
$\varphi_{\underline{x}}(\overline M)\in\mathbb R$ si
$\overline M\in C_v(X)$.

``1) $\Rightarrow$ 2)'' provient du fait que
la limite de la somme de deux suites
convergentes est \'egale \`a la somme des
limites de ces suites.

``2) $\Rightarrow$ 1)''~: Soit $\overline M$ un
\'el\'ement dans $C_v(X)$. Par l'additivit\'e de
$\varphi_{\underline{x}}$, on obtient
\[\liminf_{n\rightarrow\infty}h_{\overline
M}(x_n)=\varphi_{\underline{x}}(\overline M)=
-\varphi_{\underline{x}}(\overline
M^\vee)=-\liminf_{n\rightarrow\infty}(-h_{\overline
M
}(x_n))=\limsup_{n\rightarrow\infty}h_{\overline
M }(x_n).\]
\end{proof}

Si $\overline L $ est un fibr\'e inversible ad\'elique
sur $X$, alors le sous-ensemble
\[C_v(X,\overline
L):=\{\overline L^{\otimes n}\otimes\mathcal O_v(f)\mid
n\geqslant 1,\,\mathcal O_v(f)\in C_v(X) \}\] de
$\widehat{\Pic}(X)$ est un semi-groupe. On observe que,
si $\varphi_{\underline{x}}(\overline L)\in\mathbb R$,
alors la fonction $\varphi_{\underline{x}}$ est finie
sur $C_{v}(X,\overline L)$.

\begin{theo}\label{Thm:crietere equidistribution} Si une suite $\underline{x}=(x_n)_{n\geqslant 1}$
de points alg\'ebriques dans $X$ satisfait \`a la
condition d'\'equidistribution, alors pour tout fibr\'e
vectoriel ad\'elique $\overline L$ tel que
$\varphi_{\underline{x}}(\overline L)\in\mathbb R$, la
fonction $\varphi_{\underline{x}}$, consid\'er\'ee
comme une fonction sur $C_v(X,\overline L)$, est
diff\'erentiable en $\overline L$ pour les directions
dans $C_v(X)$. R\'eciproquement, pour que la suite
$\underline{x}$ satisfasse \`a la condition
d'\'equidistribution, il suffit qu'il existe un
$\overline L\in\widehat{\Pic}(X)$ tel que
$(h_{\overline{L}}(x_n))_{n\geqslant 1}$ converge dans
$\mathbb R$ et que la fonction
$\varphi_{\underline{x}}$ soit diff\'erentiable en
$\overline L$ pour les directions dans $C_v( X)$.
\end{theo}
\begin{proof}
``$\Longrightarrow$''~: Pour tout entier $m\geqslant 1$
et tout $\overline M\in C_v(X)$, on a
\begin{equation}\label{Equ:additivite de h}\varphi_{\underline x}(\overline
L^{\otimes m}\otimes
M)=\liminf_{n\rightarrow\infty}\big(mh_{\overline
L }(x_n)+h_{\overline M}(x_n)\big)
=m\varphi_{\underline x}(\overline
L)+\varphi_{\underline{x}}(\overline M),
\end{equation} o\`u on a utilis\'e l'hypoth\`ese de
convergence de $(h_{\overline
M}(x_n))_{n\geqslant 1}$. Donc $D_{\overline
L}\varphi_{\underline x}=\varphi_{\underline
x}$ est additive.

``$\Longleftarrow$''~: Pour tout entier $m\geqslant 1$
et tout $\overline M\in C_v(X)$, on a encore
\eqref{Equ:additivite de h}, mais cette fois-ci, on
utilise la convergence de $(h_{\overline
L}(x_n))_{n\geqslant 1}$. Par cons\'equent,
$D_{\overline L}\varphi_{\underline x}$ s'identifie \`a
la restriction de $\varphi_{\underline x}$ \`a
$C_v(X)$. On en d\'eduit que $\varphi_{\underline{x}}$
est additive en $C_v(X)$. D'apr\`es la proposition
\ref{Pro:equidistribution equivalent a ladditive}, la
suite $\underline{x}$ satisfait \`a la condition
d'\'equidistribution.
\end{proof}

\section{Minorations de la limite inf\'erieure des hauteurs}

On a interpr\'et\'e la condition d'\'equidistribution
par la diff\'erentiabilit\'e de la fonction d\'efinie
par la limite inf\'erieure d'hauteurs (le th\'eor\`eme
\ref{Thm:crietere equidistribution}). D'apr\`es la
proposition \ref{Pro:critere de differentiabilite},
cette diff\'erentiabilit\'e peut \^etre justifi\'ee par
minorer la fonction limite inf\'erieure par des
fonctions diff\'erentiables. Dans ce paragraphe, on
consid\`ere le cas o\`u la suite de points alg\'ebrique
est {\it g\'en\'erique}, autrement dit, toute
sous-vari\'et\'e ferm\'ee autre que la vari\'et\'e
totale ne contient qu'un nombre fini de points de la
suite.

\subsection{Rappels sur l'in\'egalit\'e de
pentes}\label{SubSec:rappel}

Dans ce sous-paragraphe, on rappelle quelques notions
et r\'esultats dans la th\'eorie des pentes en
g\'eom\'etrie d'Arakelov due \`a Bost. Les
r\'ef\'erences sont
\cite{BostBour96,Bost2001,Chambert,BostICM,Gaudron07}.

On appelle {\it fibr\'e vectoriel ad\'elique} sur
$\Spec K$ toute donn\'ee
$(E,(\|\cdot\|_v)_{v\in\Sigma})$, o\`u $E$ est un
espace vectoriel de rang fini sur $K$ et $\|\cdot\|_v$
est une norme sur l'espace $E\otimes_K{\mathbb C_v}$,
soumise aux conditions suivantes~:
\begin{enumerate}[1)]
\item il existe un $\mathcal O_K$-module projectif $\mathcal
E$ tel que $\mathcal E_K=E$, et que, pour tout
$\mathfrak p\in\Sigma_f$, la norme
$\|\cdot\|_{\mathfrak p}$ soit induite par la structure
de $\mathcal O_K$-module sur $\mathcal E$.
\item les normes
$(\|\cdot\|_\sigma)_{\sigma\in\Sigma_\infty}$ sont
invariantes par la conjugaison complexe.
\end{enumerate}

\'Etant donn\'es deux fibr\'es vectoriels hermitiens
$\overline E$ sur $\Spec K$, la {\it caract\'eristique
d'Euler-Poincar\'e} de $\overline E$ est par
d\'efinition
\[\chi(\overline E):=\log(\mathrm{vol}(\mathbb B(\overline E)))-
\log(\mathrm{covol}(E)),\] o\`u $\mathrm{vol}$ est une
mesure de Haar quelconque sur $E\otimes_K\mathbb
A_{K}$, et $\mathrm{covol}$ est le covolume pour la
mesure $\mathrm{vol}$ du r\'eseau $E$ dans
$E\otimes_K\mathbb A_{K}$, $\mathbb A_K$ \'etant
l'anneau des ad\`eles de $K$.

Si $E$ est non-nul, le {\it degr\'e d'Arakelov} de
$\overline E$ est
\[\widehat{\deg}(\overline E):=\chi(\overline E)-\chi(\overline K^{\rang E}),\]
o\`u $\overline K$ est le fibr\'e inversible ad\'elique
trivial sur $\Spec K$, et pour tout entier $n\geqslant
1$, $\overline K^n$ d\'esigne la somme directe
orthogonale de $n$ copies de $\overline K$. En
particulier, $\widehat{\deg}(\overline K^n)=0$ pour
tout $n$. Par convention, le degr\'e d'Arakelov du
fibr\'e vectoriel ad\'elique nul est z\'ero. Si
$\overline E$ est un fibr\'e ad\'elique hermitien
non-nul sur $\Spec K$, on appelle {\it pente} de
$\overline E$ le nombre r\'eel
\[\widehat{\mu}(\overline E):=\frac{1}{[K:\mathbb Q]}\frac{
\widehat{\deg}(\overline E)}{\rang E}.\] La {\it pente
maximale} de $\overline E$ est par d\'efinition la
valeur maximale des pentes des sous-fibr\'es
ad\'eliques hermitiens de $\overline E$, not\'e
$\widehat{\mu}_{\max}(\overline E)$. On d\'efinit
$\widehat{\mu}_{\max}(0)=-\infty$ par convention.
Soient $\overline E$ et $\overline F$ deux fibr\'es
ad\'eliques hermitiens sur $\Spec K$, et
$f:E\rightarrow F$ un homomorphisme. Pour toute
$v\in\Sigma$, on d\'esigne par $h_v(f)$ le logarithme
de la norme (d'op\'erateur) de l'application
$f_{\mathbb C_v}:E\otimes_K\mathbb C_v\rightarrow
F\otimes_K\mathbb C_v$. On note en outre
\[h(f)=\frac{1}{[K:\mathbb Q]}\sum_{v\in\Sigma}
n_vh_v(f).\]

L'in\'egalit\'e de pente suivante, qui relie les pentes
maximales de la source et du but d'un homomorphisme
injectif de fibr\'es vectoriels ad\'eliques, sera
utilis\'ee plus loin dans la minoration des invariants
arithm\'etiques.

\begin{prop}
Soient $\overline E$ et $\overline F$ deux fibr\'es
vectoriels ad\'eliques non-nuls, $f:E\rightarrow F$ une
application $K$-lineaire injective. Alors on a
l'in\'egalit\'e suivante~:
\begin{equation}\label{Equ:pente}
\widehat{\mu}_{\max}(\overline E)\leqslant
\widehat{\mu}_{\max}(\overline F)+h(f).
\end{equation}
\end{prop}
\begin{proof}
Voir \cite[Lemme 6.4]{Gaudron07} pour la
d\'emonstration.
\end{proof}

\subsection{Invariants asymptotiques des fibr\'es inversibles hermtiens}
Soient $\pi: X\rightarrow\Spec K$ un $K$-sch\'ema
projectif et int\`egre. Des invariants arithm\'etique
sont naturellement d\'efinis pour les fibr\'es
inversibles hermitiens sur $X$. Dans la suite, on
pr\'esente des constructions classiques et introduit
quelques notions.

\subsubsection*{Minimum essentiel}Soit $\overline{L}$ un fibr\'e inversible
hermitien sur $X$. Le {\it minimum essentiel} de
$\overline{ L}$, c'est-\`a-dire le {premier minimum}
(logarithmique) de $\overline{L}$, est par d\'efinition
\begin{equation}
\widehat{\mu}_{\mathrm{ess}}(\overline{
L }):=\sup_{\begin{subarray}{c}U\subsetneq X\\
U\text{ ouvert}\end{subarray}}\inf_{x\in U(\overline
K)}h_{\overline{ L}}(x).
\end{equation}

\begin{lemm}\label{Lem:comparaison de mu ess}
Soient $\overline L_1$ et $\overline{L}_2$ deux
fibr\'es inversibles hermitiens sur $X$. Si $\varphi$
est un homomorhpisme non-nul de $L_1$ vers $ L_2$,
alors
\[\widehat{\mu}_{\mathrm{ess}}(\overline L_1)
\leqslant\widehat{\mu}_{\mathrm{ess}}(\overline L_2
)+h(\varphi).\]
\end{lemm}
\begin{proof}
Pour tout point alg\'ebrique $x\in X(\overline K)$, on
a
\[h_{\overline L_1}(x)=h_{\overline L_2}(x)+\sum_{v\in\Sigma}
n_v\log\|f_x\|_v\leqslant h_{\overline L_2}(x)+h(f).\]
\end{proof}
\begin{prop}
Pour tout fibr\'e inversible ad\'elique $\overline{L}$
sur $X$, $\widehat{\mu}_{\mathrm{ess}}(\overline{L
})<+\infty$.
\end{prop}
\begin{proof}
Soit $\overline{\mathscr L}$ un fibr\'e inversible
ad\'elique ample tel que $\overline
L^\vee\otimes\overline{\mathscr L}$ soit effectif.
D'apr\`es le lemme \ref{Lem:comparaison de mu ess}, on
a $\widehat{\mu}_{\mathrm{ess}}(\overline
L)\leqslant\widehat{\mu}_{\mathrm{ess}}(\overline{\mathscr
L })$. La finitude de
$\widehat{\mu}_{\mathrm{ess}}(\overline L)$ provient
donc de celle de
$\widehat{\mu}_{\mathrm{ess}}(\overline{\mathscr L})$,
\'etablie dans \cite{Zhang95}.
\end{proof}

\subsubsection*{Pente maximale asymptotique} Soit $\overline
L$ un fibr\'e inversible ad\'elique tel que $L$ est
gros (c'est-\`a-dire que la dimension d'Iikata de $X$
relativement \`a $L$ est maximale, voir \cite[\S
2.2]{LazarsfeldI}). La {\it pente maximale
asymptotique} de $\overline L$ introduite dans
\cite{Chen08} est un invariant arithm\'etique qui
contr\^ole le comportement asymptotique de la norme de
la plus petite section effective non-nulle de
$\overline L^{\otimes n}$ lorsque $n$ tend vers
l'infinie. Elle est d\'efinie comme
\[\widehat{\mu}^\pi_{\max}(\overline L):=\lim_{n\rightarrow+\infty}
\frac 1n\widehat{\mu}_{\max}(\overline L^{\otimes
n}),\] o\`u l'existence de la limite est justifi\'ee
dans \cite[th\'eor\`eme 4.1.8]{Chen08}.

\subsubsection*{Capacit\'e sectionnelle} La capacit\'e
sectionnelle d'un fibr\'e inversible ad\'elique
$\overline L$ sur $X$ est par d\'efinition
\begin{equation}\label{Equ:capacite sectionnelle}S(\overline L):=\lim_{n\rightarrow+\infty}
\frac{\chi(\pi_*\overline L^{\otimes
n})}{n^{d+1}/(d+1)!}\in[-\infty,+\infty[,\end{equation}
o\`u $d$ est la dimension de $X$. C'est une notion qui
g\'en\'eralise le nombre d'auto-intersection
arithmetique~: si $L$ est ample et si toutes les
m\'etriques $\|\cdot\|_v$ sont semi-positive, alors
$S(\overline L)$ est finie et \'egale \`a
$\widehat{c}_1(\overline L)^{d+1}$. C'est une
cons\'equence du th\'eor\`eme de Hilbert-Samuel
d\'emontr\'e par
\cite{Gillet-Soule,Abbes-Bouche,Zhang95,Autissier01,Randriam06}
dans diverses situations. L'existence de la limite
\eqref{Equ:capacite sectionnelle} est d\'emontr\'ee par
Rumely, Lau et Varley \cite{Rumely_Lau_Varley} dans le
cas o\`u $L$ ample, et puis par le pr\'esent auteur
\cite{Chen_bigness} au cas g\'en\'eral. Il s'av\`ere
que la capacit\'e sectionnelle est toujours nulle
lorsque $L$ n'est pas gros, car on a
$\chi(\pi_*\overline L^{\otimes
n})=O(n^{\kappa(X,L)+1})$, o\`u $\kappa(X,L)$ est la
dimension d'Iitaka de $X$ relativement \`a $L$ (voir
\cite[corollaire 2.1.38]{LazarsfeldI}). Si $L$ est
gros, on note
\begin{equation}\widehat{\mu}^\pi(\overline
L):=\frac{S(\overline L)}{ [K:\mathbb
Q](d+1)\mathrm{vol}(L)},\end{equation} o\`u le {\it
volume} de $L$ est d\'efini comme la limite
\begin{equation*}\mathrm{vol}(L):=\lim_{n\rightarrow+\infty}\frac{\rang_KH^0(X,L^{\otimes
n})}{ n^d/d!}.\end{equation*} On appelle
$\widehat{\mu}^\pi(\overline L)$  la {\it pente
asymptotique} de $\overline L$.

\subsubsection*{Volume arithm\'etique et pente postive asymptotique}
Le volume arithm\'etique d'un fibr\'e inversible
ad\'elique $\overline L$ sur $X$ est d\'efini dans
\cite{Moriwaki07} comme
\begin{equation}\label{Equ:volume de moriwaki}\widehat{\mathrm{vol}}(\overline L)=\limsup_{n\rightarrow\infty}
\frac{\widehat{h}^0(\overline L^{\otimes
n})}{n^{d+1}/(d+1)!},\end{equation} o\`u
\[\widehat{h}^0(\overline L^{\otimes n}):=\log\#
\Big\{s\in H^0(X,L^{\otimes n})\;\Big|\;\forall\,
v\in\Sigma,\,\|s\|_{v,\sup}\leqslant 1\Big\}.\] Il
s'agit en fait d'une limite dans \eqref{Equ:volume de
moriwaki}, voir \cite{Chen_bigness}. Si $L$ est gros,
on note
\[\widehat{\mu}_+^\pi(\overline L):=\frac{\widehat{\mathrm{vol}}(\overline L)}{
[K:\mathbb Q](d+1)\mathrm{vol}(L)}.\] On appelle
$\widehat{\mu}_+^\pi(\overline L)$ la {\it pente
positive asymptotique} de $\overline L$.

\subsubsection*{Mesure asymptotique} Soit $\overline L$
un fibr\'e inversible ad\'elique sur $X$ tel que $L$
est gros. La mesure asymptotique de $\overline L$,
introduite dans \cite{Chen08} et \cite{Chen_bigness},
est un invariant tr\`es g\'en\'eral qui permet de
retrouver divers invariants arithm\'etriques de
$\overline L$ par int\'egration. Soit $n\geqslant 1$ un
entier. L'espace vectoriel $E_n:=\pi_*(L^{\otimes n})$
est filtr\'e par ses minima successifs. Pour tout
$t\in\mathbb R$, on note
\[\mathcal F_tE_n:=\mathrm{Vect}_K(\{s\in E_n\mid\forall\, \mathfrak p\in\Sigma_f,\,
\|s\|_{\mathfrak p,\sup}\leqslant
1,\,\forall\,\sigma\in\Sigma_\infty,\,\|s\|_{\sigma,\sup}\leqslant
e^{-nt} \}).\] La suite de mesure de probabilit\'e
$\big(-\frac{\mathrm{d}}{\mathrm{d}t}\rang(\mathcal
F_tE_n)/\rang(E_n)\big)_{n\geqslant 1}$ converge
vaguement vers une mesure de probabilit\'e bor\'elienne
$\nu_{\overline L}$ sur $\mathbb R$ (voir
\cite{Chen_Fujita} corollaire 3.13 et remarque
3.14), les d\'eriv\'ees \'etant prises au sens de
distribution. On l'appelle la {\it mesure asymptotique}
de $\overline L $. Cette mesure peut aussi \^etre
construite par la filtration de Harder-Narasimhan. Plus
pr\'ecis\'ement, si on note $\mathcal G$ la filtration
de $E_n$ telle que
\[\forall\,t\in\mathbb R,\quad\mathcal G_tE_n=\sum_{\begin{subarray}{c}
F\subset E_n\\
\widehat{\mu}_{\min}(\overline F)\geqslant nt
\end{subarray}}F,\]
alors $\nu_{\overline L}$ est aussi la limite vague de
la suite de mesures
$\big(-\frac{\mathrm{d}}{\mathrm{d}t}\rang(\mathcal
G_tE_n)/\rang(E_n)\big)_{n\geqslant 1}$, voir \cite[\S
3]{Chen_Fujita} pour le d\'etail.

Comme $\mathcal G_tE_n=0$ d\`es que $t\geqslant
n^{-1}\widehat{\mu}_{\max}(\pi_*(\overline L^{\otimes n
}))$, on obtient que le support de la mesure limite
$\nu_{\overline L}$ est contenu dans l'intervalle
ferm\'e $]-\infty,\widehat{\mu}_{\max}^\pi(\overline
L)]$.

Plusieurs invariants arithm\'etiques de $\overline L$
peuvent \^etre repr\'esent\'es comme des int\'egrales
par rapport \`a $\nu_{\overline L}$. On a (cf.
\cite{Chen_bigness})
\begin{equation}\widehat{\mu}^\pi(\overline L)=\int_{\mathbb
R}x\,\nu_{\overline L}
(\mathrm{d}x),\qquad\widehat{\mu}_+^\pi(\overline L)=
\int_{\mathbb R}x_+\,\nu_{\overline
L}(\mathrm{d}x),\end{equation} o\`u $x_+=\max(x,0)$.

\subsection{Comparaisons des invariants arithm\'etiques}
Le but de ce sous-paragraphe est d'\'etablir les
comparaisons comme ci-dessous.

\begin{prop}\label{Pro:comparaison}
Soit $\overline L$ un fibr\'e inversible ad\'elique sur
$X$ tel que $L$ soit gros. Les in\'egalit\'es suivantes
sont vraies~:
\begin{equation}\label{Equ:inegalities}
\widehat{\mu}_{\mathrm{ess}}(\overline
L)\geqslant\widehat{\mu}_{\max}^\pi(\overline
L)\geqslant\widehat{\mu}^\pi(\overline L).
\end{equation}
Si de plus $\overline L$ est gros, ou de fa\c{c}on \'equivalente, $\widehat{\mu}_{\max}^\pi(\overline L)>0$, alors 
\begin{equation}\label{Equ:inegalities2}
\widehat{\mu}_{\max}^\pi(\overline
L)\geqslant\widehat{\mu}_{+}^\pi(\overline
L)\geqslant\widehat{\mu}^\pi(\overline L).
\end{equation}
\end{prop}

La deuxi\`eme in\'egalit\'e de \eqref{Equ:inegalities} et les in\'egalit\'es dans \eqref{Equ:inegalities2} proviennent
simplement de l'interpr\'etation des invariants
arithm\'etiques par la mesure asymptotique. La
d\'emonstration de la premi\`ere in\'egalit\'e fait
appel \`a la m\'ethode de pentes appliqu\'ee aux
applications d'\'evaluation, voir \cite{BostICM} pour
un survol et pour une liste compl\`ete de
r\'ef\'erences.

\begin{lemm}\label{Lem:minoration de mu ess}
Soit $\overline L$ un fibr\'e inversible ad\'elique sur $X$.
Si $\mathcal B\subset X(\overline K)$ est une famille de
points qui est Zariski dense dans $X(\overline K)$,
alors \[\sup\{h_{\overline L}(P)\;|\;P\in\mathcal
B\}\ge\frac 1n\widehat{\mu}_{\max}(\pi_*(\overline
L^{\otimes n}))\] quel que soit $n\in\mathbb N^*$. En
particulier, on a $\sup\{h_{\overline
L}(P)\;|\;P\in\mathcal
B\}\ge\widehat{\mu}_{\max}^\pi(\overline L)$ si $L$ est gros.
\end{lemm}
\begin{proof}
Pour tout entier $n\geqslant 1$, soit $\varphi_n$
l'application d'\'evaluation
\[H^0(X,L^{\otimes n})_{\overline K}
\longrightarrow\bigoplus_{P\in\mathcal B}P^*L^{\otimes
n}.\] Comme $\mathcal B$ est Zariski dense, $\varphi_n$
est injective. Par cons\'equent, il existe un
sous-ensemble fini $\mathcal B_n$ de $\mathcal B$ tel
que $p_n\varphi_n$ soit un isomorphisme, o\`u
\[p_n:\bigoplus_{P\in\mathcal B}P^*L^{\otimes
n}\longrightarrow \bigoplus_{P\in\mathcal B_n}P^*L^{\otimes
n}\] est la projection canonique. L'in\'egalit\'e de
pentes \eqref{Equ:pente} montre alors que \[
\widehat{\mu}_{\max}(\pi_*(\overline L^{\otimes n}))\le
\max_{P\in\mathcal B_n}nh_{\overline
L}(P)+\log\rang(\pi_*L^{\otimes n}).\] Par passage \`a
la limite on obtient l'in\'egalit\'e souhait\'ee.
\end{proof}

\begin{proof}[D\'emonstration de la proposition \ref{Pro:comparaison}]
Soit $\nu_{\overline L}$ la mesure asymptotique de
$\overline L$. Comme le support de $\nu_{\overline L}$
est born\'e sup\'erieurement par
$\widehat{\mu}_{\max}^{\pi}(\overline L)$, on obtient
\[\widehat{\mu}_{\max}^\pi(\overline L)\geqslant\int_{\mathbb
R}x\,\nu_{\overline L}(\mathrm{d}x),\]
d'o\`u
$\widehat{\mu}^\pi_{\max}(\overline
L)\geqslant\widehat{\mu}^\pi(\overline L)$.
Si de plus $\widehat{\mu}_{\max}^\pi>0$, on a 
\[\widehat{\mu}_{\max}^\pi(\overline L)\geqslant\int_{\mathbb R}
x_+\,\nu_{\overline
L}(\mathrm{d}x)\geqslant\int_{\mathbb
R}x\,\nu_{\overline L}(\mathrm{d}x),\] et donc
$\widehat{\mu}^\pi_{\max}(\overline
L)\geqslant\widehat{\mu}_{+}^\pi(\overline
L)\geqslant\widehat{\mu}^\pi(\overline L)$. Il reste
\`a v\'erifier la premi\`ere in\'egalit\'e de \eqref{Equ:inegalities}.

Soit $t>\widehat{\mu}_{\mathrm{ess}}(\overline L)$ et soit
$\mathcal B$ l'ensemble des points dans $X(\overline
K)$ dont la hauteur est inf\'erieure ou \'egale \`a $t$. Par d\'efinition la
famille $\mathcal B$ est Zariski dense dans
$X(\overline K)$. Le lemme \ref{Lem:minoration de mu
ess} montre alors que
$t\geqslant\widehat{\mu}^{\pi}_{\max}(\overline L)$. Comme
$t$ est arbitraire, on obtient
$\widehat{\mu}_{\mathrm{ess}}(\overline
L)\geqslant\widehat{\mu}^\pi_{\max}(\overline L)$.
\end{proof}

\begin{rema}\label{Rem:egalite}
On observe dans la d\'emonstration de la proposition
\ref{Pro:comparaison} que
$\widehat{\mu}_{\max}^\pi(\overline
L)=\widehat{\mu}_+^\pi(\overline L)$ implique que la
mesure $\nu_{\overline L}$ se r\'eduit \`a la mesure de
Dirac concentr\'ee en
$\widehat{\mu}_{\max}^\pi(\overline L )$. Cette
derni\`ere condition est \'equivalente \`a
l'\'egalit\'e $\widehat{\mu}_{\max}^\pi(\overline
L)=\widehat{\mu}^\pi(\overline L)$. Dans ce cas-l\`a,
le polygone de Harder-Narasimhan asymptotique de
$\overline L$ (voir \cite{Chen08}) se r\'eduit \`a un
morceau de droite, et on dit que $\overline L$ est {\it
asymptotiquement semi-stable}.
\end{rema}

\section{Applications}

Dans ce paragraphe, on d\'eveloppe des applications du
crit\`ere d'\'equidistribution \'etabli dans
\S\ref{Sec:critere}. D'abord, on explique comment
interpr\'eter des r\'esultats classiques via le
principe variationnel par la d\'erivabilit\'e de
certain invariant arithm\'etique. Ensuite, on
d\'emontre un r\'esultat d'\'equidistribution en
utilisant la d\'erivabilit\'e de la pente positive
asymptotique, \'etablie dans \cite{Chen_Diff}.
\subsection{Interpr\'etation des r\'esultats classiques}
Soient $\pi:X\rightarrow\Spec K$ un $K$-sch\'ema
projectif et int\`egre et $\overline L$ un fibr\'e
inversible ad\'elique sur $X$. Soit
$\underline{x}=(x_n)_{n\geqslant 1}$ une suite de
points alg\'ebriques dans $X$. Dans les r\'esultats
d'\'equidistribution via le principe variationnel comme
dans
\cite{Szpiro_Ullmo_Zhang,Ullmo98,Zhang98,Autissier01,Yuan07,Ber_Bou08},
les hypoth\`eses suivantes sont suppos\'ees:
\begin{enumerate}[({\bf H}1)]
\item la suite $(x_n)_{n\geqslant 1}$ est
g\'en\'erique, autrement dit, toute sous-vari\'et\'e
ferm\'ee $Y$ de $X$ qui est autre que $X$ ne contient
qu'un nombre fini de points parmi les $x_n$;
\item la suite de hauteur $(h_{\overline L}(x_n))_{n\geqslant
1}$ converge vers $\widehat{\mu}^\pi(\overline L)$.
\end{enumerate}
Pour tout fibr\'e inversible ad\'elique $\overline M$,
on note \[\varphi_{\underline{x}}(\overline M):=
\liminf_{n\rightarrow+\infty}h_{\overline M}(x_n)\in
[-\infty,+\infty] .\]

\begin{prop}
Si la suite $\underline{x}$ est g\'en\'erique, alors
$\varphi_{\underline{x}}(\overline M )\geqslant
\widehat{\mu}_{\mathrm{ess}}(\overline M)$ pour tout
$\overline M\in\widehat{\mathrm{Pic}}(X)$.
\end{prop}
\begin{proof}
Soit $U$ un sous-sch\'ema ouvert non-vide de $X$. Comme
la suite $\underline{x}$ est g\'en\'erique, on obtient
que $U$ contient tous sauf un nombre fini de points de
$\underline{x}$. Par cons\'equent,
\[\varphi_{\underline{x}}(\overline M)\geqslant\liminf_{n\rightarrow+\infty}h_{\overline L}(x_n)
\geqslant\inf_{x\in U(\overline K)}h_{\overline M}(x),
\]
d'o\`u $\varphi_{\underline{x}}(\overline
M)\geqslant\widehat{\mu}_{\mathrm{ess}}(\overline M)$.
\end{proof}

\begin{coro}\label{Cor:inegalites}
On suppose que $\overline L$ est gros. Sous les hypoth\`eses ({\bf H}1) et ({\bf H}2), on a
\[\varphi_{\underline{x}}(\overline L)=\widehat{\mu}_{\mathrm{ess}}
(\overline L)=\widehat{\mu}_{\max}^\pi(\overline
L)=\widehat{\mu}_+^\pi(\overline
L)=\widehat{\mu}^\pi(\overline L).\]
\end{coro}
\begin{proof}
C'est une cons\'equence imm\'ediate de l'in\'egalit\'e
$\varphi_{\underline{x}}(\overline
L)\geqslant\widehat{\mu}_{\mathrm{ess}}(\overline L)$
et de \eqref{Equ:inegalities}.
\end{proof}

D'apr\`es le th\'eor\`eme \ref{Thm:crietere
equidistribution} et la proposition \ref{Pro:critere de
differentiabilite}, pour que $\underline{x}$ satisfasse
\`a la condition d'\'equidistribution en $v$, il suffit
de montrer que la fonction $\widehat{\mu}^\pi$ est
diff\'erentiable en $\overline L$ pour les directions dans
$C_v(X)$. Il s'av\`ere que, dans les articles cit\'es
plus haut, les auteurs ont effectivement d\'emontr\'e
cette diff\'erentiabilit\'e et ont calcul\'e la
diff\'entielle de $\widehat{\mu}^\pi$ en $\overline L$,
bien que ces r\'esultats n'ont pas \'et\'e \'enonc\'es
de cette fa\c{c}on.

\begin{exem}
On suppose que $\overline L$ est ample au sens de
\cite{Zhang95}, alors pour tout $\overline
M\in\widehat{\Pic}(X)$ de m\'etriques lisses, le
fibr\'e inversible ad\'elique $\overline L^{\otimes
n}\otimes\overline M$ est ample lorsque $n$ est
suffisamment grand. Donc le th\'eor\`eme de
Hilbert-Samuel de Gillet et Soul\'e \cite{Gillet-Soule}
montre que \[\widehat{\mu}^\pi(\overline L^{\otimes
n}\otimes\overline M)=\frac{(n\widehat{c}_1(\overline
L)+\widehat{c}_1(\overline M))^{d+1}}{[K:\mathbb Q
](d+1)(n{c_1}( L)+c_1(M))^d}\] Par cons\'equent, on a
\[D_{\overline L}\widehat{\mu}^\pi(\overline M)=
\frac{\widehat{c}_1(\overline
L)^d\widehat{c}_1(\overline M)}{[K:\mathbb Q]c_1(L)^d}-
d\widehat{\mu}^\pi(\overline
L)\frac{c_1(L)^{d-1}c_1(M)}{c_1(L)^d}.
\]
Cette expression est additive par rapport \`a
$\overline M$. De plus, si $\overline M$ est de la
forme $\mathcal O_\sigma(f)$
($\sigma:K\rightarrow\mathbb C$), o\`u $f$ est une
function lisse $X_{\sigma}^{\mathrm{an}}$, alors
\[D_{\overline L}\widehat{\mu}^\pi(\mathcal O_{\sigma}(f))
=\frac{\widehat{c}_1(\overline
L)^d\widehat{c}_1(\mathcal O_\sigma(f))}{[K:\mathbb
Q]c_1(L)^d}.\] Cela montre que
$\frac{\widehat{c}_1(\overline L )^d}{[K:\mathbb
Q]c_1(L)^d}$ peut \^etre consid\'er\'ee comme une mesure de probabilit\'e de Radon
sur $X_{\sigma}^{\mathrm{an}}$ qui est la mesure limite
de $(\eta_{x_n,\sigma})_{n\geqslant 1}$.
\end{exem}

Depuis l'article \cite{Szpiro_Ullmo_Zhang}
o\`u cette strat\'egie a \'et\'e propos\'ee, on cherche
\`a affaiblir la condition de positivit\'e de
$\overline L$. On r\'esume quelques exemples de tels r\'esultats en les \'enon\c{c}ant dans le langage de differentiabilit\'e.
\begin{enumerate}[1)]
\item Dans \cite{Autissier01}, Autissier a montr\'e que, si $X$ est de dimension $1$, alors la fonction $\widehat{\mu}^\pi$ est diff\'erentiable en tout $\overline L$ tel que $L$ est ample pour les directions dans $C_\sigma(X)$.
\item Dans \cite{Chambert-Loir_Thuillier}, en utilisant l'analogue de l'in\'egalit\'e de Siu en g\'eom\'etrie d'Arakelov \cite{Yuan07}, les auteurs ont montr\'e que, si $L$ est ample et si les m\'etriques sur $L$ sont semi-positives, alors la fonction $\widehat{\mu}^\pi$ est diff\'erentiable pour les directions $\overline M$ int\'egrables (c'est-\`a-dire que $\overline M$ s'\'ecrit comme la diff\'erence de deux fibr\'es inversibles ad\'eliques amples).
\item Dans \cite{Ber_Bou08}, en utilisant le noyau de Bergman, les auteurs ont montr\'e que, si $L$ est gros et si $\widehat{\mu}^\pi(\overline L)$ est fini, alors la fonction $\widehat{\mu}^\pi$ est diff\'erentiable pour les directions dans $C_\sigma(X)$, o\`u $\sigma\in\Sigma_\infty$.
\end{enumerate}

\subsection{Un r\'esultat d'\'equidistribution}

Dans ce sous-paragraphe, on d\'emontre le r\'esultat suivant~:
\begin{theo}\label{Thm:equidistr}
On suppose que $\overline L$ est gros. Si la suite $\underline{x}$ est g\'en\'erique et si la suite de hauteurs $(h_{\overline L}(x_n))_{n\geqslant 1}$ converge vers $\widehat{\mu}_+^\pi(\overline L)$, alors la suite $\underline{x}$ v\'erifie la condition d'\'equidistribution.
\end{theo}
\begin{proof}
D'apr\`es \cite[th\'eor\`eme 1.1]{Chen_Diff}, la fonction $\widehat{\mu}_+^\pi$ est diff\'erentiable en $\overline L$ pour les directions dans $\widehat{\Pic}(X)$. Par cons\'equent, la suite $\underline{x}$ satisfait \`a la condition d'\'equidistribution, compte tenu du th\'eor\`eme \ref{Thm:crietere equidistribution}.
\end{proof}

\begin{rema}
Dans le th\'eor\`eme \ref{Thm:equidistr}, on ne suppose aucune condition de positivit\'e sur les m\'etriques de $\overline L$.
\end{rema}

\begin{coro}
On suppose que $L$ est gros. Si la suite $\underline{x}$ est g\'en\'erique et si la suite de hauteurs $(h_{\overline L}(x_n))_{n\geqslant 1}$ converge vers $\widehat{\mu}^\pi(\overline L)$, alors la suite $\underline{x}$ satisfait \`a la condition d'\'equidistribution.
\end{coro}
\begin{proof}
Comme $L$ est gros, il existe un fibr\'e ad\'elique $\overline M$ sur $\Spec\mathcal O_K$ tel que $\overline N:=\overline L\otimes\pi^*\overline M$ soit gros (cf. \cite[remarque 4.2.16]{Chen08}). Pour tout entier $n\geqslant 1$, on a $h_{\overline N}(x_n)=h_{\overline L}(x_n)+\widehat{\mu}(\overline M)$. D'autre part, on a $\widehat{\mu}^\pi(\overline N)=\widehat{\mu}^\pi(\overline L)+\widehat{\mu}(\overline M)$. Donc $(h_{\overline N}(x_n))_{n\geqslant 1}$ converge vers $\widehat{\mu}^\pi(\overline N)$. D'apr\`es le corollaire \ref{Cor:inegalites}, on obtient $\widehat{\mu}^\pi(\overline N)=\widehat{\mu}^\pi_+(\overline N)$. D'o\`u $\underline{x}$ satisfait \`a la condition d'\'equidistribution, compte tenu du th\'eor\`eme \ref{Thm:equidistr}.
\end{proof}

\begin{rema}
Le th\'eor\`eme \ref{Thm:equidistr} est en fait une nouvelle d\'emonstration des r\'esultats classiques. On suppose que $\overline L$ est gros. Bien que en g\'en\'eral $\widehat{\mu}_+^\pi(\overline L)$ est plus grand ou \'egal \`a $\widehat{\mu}^\pi(\overline L)$, mais l'hypoth\`ese $\varphi_{\underline{x}}(\overline L)=\widehat{\mu}_+^\pi(\overline L)$ force ces deux quantit\'es \`a \^etre \'egales, comme montr\'e dans la remarque \ref{Rem:egalite}. D'autre part, l'hypoth\`ese ({\bf H}2) semble \^etre pr\'evil\'egi\'ee aux applications que l'on dispose. Par exemple, on observe dans \cite{Sombra05} que le minimum essentiel d'un espace projectif est toujours strictement plus grand que sa hauteur normalis\'ee. Plus pr\'ecis\'ement, si $X=\mathbb P_K^n$ et si $\overline L$ est $\mathcal O_X(1)$ muni des m\'etriques de Fubini-Study, alors on a
\[\widehat{\mu}_{\mathrm{ess}}(\overline L)=\frac{1}{2}\log(n+1),\qquad \widehat{\mu}^\pi(\overline L)=\frac{1}{2}\sum_{j=2}^{n+1}\frac 1j.\]
Cette observation conduit \`a la question suivante~: est-ce que d'autre minoration du minimum essentiel permet d'avoir un \'enonc\'e d'\'equidistribution plus g\'en\'eral? Un candidat est sans doute la pente maximale asymptotique $\widehat{\mu}^\pi_{\max}$. Voici alors une question plus pr\'ecise~: quel est le domain de diff\'erentiabilit\'e de la fonction $\widehat{\mu}_{\max}^\pi$?
\end{rema}

\backmatter
\bibliography{chen}
\bibliographystyle{smfplain}

\end{document}